\documentclass{amsart}
\usepackage{amsmath,amssymb,url}

\newtheorem{theorem}{Theorem}[section]
\newtheorem{lemma}[theorem]{Lemma}
\newtheorem{de}[theorem]{Definition}

\newtheorem{pr}[theorem]{Proposition}

\newtheorem{qu}[theorem]{Question}

\newtheorem{co}[theorem]{Corollary}

\def \from{\leftarrow}

\def \k{\kappa}

\def \~{\, \tilde {}\,}

\def \l{\lambda}

\def \w{\omega}

\def\0{\emptyset}
\def\ms{\setminus}
\def\drop[1]{ }

\def\cl{\overline}
\def\Int{\mathrm{Int}}

\def\Toq{\Rightarrow \kern-12pt^?}

\def\prc{\prec\kern-8.5pt\circ\ }
\def\cl{\overline}
\def\Int{\mathrm{Int}}

\def\Toq{\Rightarrow \kern-12pt^?}

\def\prc{\prec\kern-8.5pt\circ\ }

\def \clo{\mathrm{cl}}

\def\conv{\mathrm{conv}}

\begin{document}

\title{Monotone Covering Properties 
defined by Closure-Preserving Operators}

\author{Strashimir G. Popvassilev}
\author{John E. Porter}

\address[Strashimir G. Popvassilev]{ 
The City College and Medgar Evers College, CUNY, New York, NY, United States (part time), and
Institute of Mathematics, Bulgarian Academy of Sciences, Sofia, Bulgaria (on leave)
}

\address[John E. Porter]{ Department of Mathematics and Statistics, Faculty Hall 6C,
 Murray State University, Murray, KY 42071-3341, United States}

\begin{abstract}
We continue Gartside, Moody, and Stares' study of versions of
monotone paracompactness.
We show that the class of spaces with a monotone
closure-preserving open operator is strictly larger
than those with a monotone open locally-finite operator.
We prove that monotonically metacompact GO-spaces
have a monotone open locally-finite operator, and
so do GO-spaces with a monotone (open or not)
closure-preserving operator, whose underlying LOTS
has a $\sigma$-closed-discrete dense subset.
A GO-space with a $\sigma$-closed-discrete
dense subset and a monotone closure-preserving
operator is metrizable. A compact LOTS with a monotone open closure-preserving operator is metrizable.
\end{abstract}

\maketitle

\vskip 0.2in

\begin{center}
{\em Dedicated to our mentor and dear friend Gary Gruenhage with gratitude
and well-wishes on the occasion of his 70-th birthday. 
}
\end{center}

\section{Introduction}

Gartside and Moody \cite[Theorem 1]{GM} proved that
a space is protometrizable
if and only if the space has a monotone star-refinement
operator\footnote{A monotone star-refinement operator is a
function $r:\mathcal{C}\to \mathcal{C}$ (where $\mathcal{C}$
is the set of all open covers of $X$) such that
(1) for every $\mathcal{U}\in \mathcal{C}$,
$r(\mathcal{U})$ star-refines $\mathcal{U}$, and
(2) if $\mathcal{U},\mathcal{V}\in \mathcal{C}$ and
$\mathcal{U}$ refines $\mathcal{V}$,
then
$r(\mathcal{U})$ refines $r(\mathcal{V})$.},
and asked whether the class of protometrizable spaces coincided
with the class of spaces with a monotone open locally-finite
operator.

\begin{de}\cite{GM} A monotone open locally-finite operator is a
function $r:\mathcal{C}\to \mathcal{C}$, where $\mathcal{C}$
is the set of all open covers of $X$, such that
(1) for every $\mathcal{U}\in \mathcal{C}$,
$r(\mathcal{U})$ is a locally-finite open refinement of $\mathcal{U}$,
and (2) if $\mathcal{U},\mathcal{V}\in \mathcal{C}$ and
$\mathcal{U}$ refines $\mathcal{V}$,
then
$r(\mathcal{U})$ refines $r(\mathcal{V})$.
\end{de}

In \cite{S}, Stares showed that different characterizations
of paracompact spaces, when monotonized, may
give rise to different classes of spaces, and asked
which monotonized characterizations coincide.
The authors showed in \cite{PP} that the class of spaces
with a monotone open locally-finite operator is strictly
larger than the class of protometrizable spaces.

It is well-known (E. Michael, \cite[5.1.G]{E})
that if every open cover of a regular T$_1$
space $X$ has a closure preserving refinement
(of arbitrary sets), then $X$ is paracompact,
i.e. every open cover has an {\em open}
locally-finite (and hence {\em open} closure-preserving)
refinement. Recall that a family $\mathcal F$ of subsets of
a space $X$ is called {\em closure-preserving}
if $\overline{\cup \mathcal H} =
\cup \{\overline{H} : H \in \mathcal H\}$
for every subfamily
$\mathcal H \subseteq \mathcal F$.
Extending 
Gartside, Moody, and Stares' study, 
we explore spaces with a monotone closure-preserving operator.

\begin{de}
A space $X$ is said to have a monotone
closure-preserving operator $r$
if for every open cover $\mathcal U$
we have that $r(\mathcal U)$ is
a closure-preserving cover (of arbitrary sets)
that refines $\mathcal U$, and
if $\mathcal U,\mathcal V$ are open covers such that
$\mathcal U$ refines $\mathcal V$, then
$r(\mathcal U)$ refines $r(\mathcal V)$.
If, in addition, each $U\in r(\mathcal U)$
is required to be
open then $r$ is called a monotone open
closure-preserving operator for $X$.
\end{de}

In a similar manner, we can distinguish between
monotone (not necessarily open) locally-finite
operators and monotone open locally-finite operators.
We do not know if the existence of
such monotone operators
necessarily implies the existence of open ones.
We provide partial answers to the following two questions in
Theorem~\ref{listmore} and Theorem~\ref{GOopen}.

\begin{qu}\label{opencp}
(a) If a space $X$ has a monotone closure-preserving operator,
must it have a monotone closure-preserving open operator?
(b) What if $X$ is a GO-space, or (c) if $X$ is a compact LOTS?
\end{qu}

\begin{qu}\label{openlf}
If a space has a monotone
(not necessarily open) locally-finite operator,
must it have a monotone open locally-finite operator?

\end{qu}

The following proposition shows that
spaces with a monotone open closure-preserving operator form
a broader class than those with
a monotone open locally-finite operator.

\begin{theorem}
Any space with only one non-isolated point has a
monotone open closure-preserving operator.
\end{theorem}

\begin{proof}
Suppose $p$ is the only non-isolated point of a topological space $X$.  If $\mathcal{U}$ is an open cover of $X$,
let $\mathcal U_p=\{U\in\mathcal{U}:p\in U\}$.
It is easy to check that
$r(\mathcal{U})=\mathcal U_p
\cup \{\{x\}:x\in X\setminus \bigcup
\mathcal U_p\}$ is the required monotone open closure-preserving operator.
\end{proof}

\begin{de} \cite{BHL,Pmc}
A space is monotonically
(countably) metacompact if there is a function
$r$ that assigns to each (countable) open cover
$\mathcal{U}$ of $X$
a point-finite open refinement $r(\mathcal{U})$ covering $X$
such that if $\mathcal{V}$ is a (countable) open cover of
$X$ and $\mathcal{U}$ refines $\mathcal{V}$,
then $r(\mathcal{U})$ refines $r(\mathcal{V})$.
\end{de}

Clearly a monotone open locally-finite operator is both
closure-preserving and point-finite.
The one-point compactification 
of an uncountable discrete set of size $\k\ge\omega_1$ and
the Sequential Fan are examples of spaces that
have a monotone open closure-preserving
operator but no monotone locally-finite open operator
nor point-finite open operator
(see \cite[Theorem 2.4]{CG} and \cite[Theorem 3.2]{CG2},
respectively). This shows the metrization theorems
of Chase and Gruenhage \cite{CG,CG2} on compact or separable
monotonically metacompact spaces do not extend to
spaces with monotone open closure-preserving operators.

However, GO-spaces with a monotone closure-preserving
operator do behave similarly to monotonically metacompact GO-spaces.
By modifying results in \cite{BHL} and \cite{PL},
in Section \ref{lambdadense} we show that GO-spaces with a
monotone closure-preserving operator are monotonically metacompact
when the underlying LOTS 
has a $\sigma$-closed-discrete dense subset.
Moreover, monotonically metacompact
GO-spaces have a monotone open locally-finite operator.

Section \ref{esep} is devoted to metrization results. GO-spaces with a $\sigma$-closed-discrete dense
subset and a monotone closure-preserving operator are metrizable,
while every compact LOTS
with a monotone open closure-preserving operator is metrizable. These results should be compared to those
on monotonically (countably) metacompact spaces in \cite{BHL} and \cite{PL}.

\section{Monotone Operators in GO-Spaces}\label{lambdadense}

For any GO-space $(X, \tau, <)$, we use following notation found in \cite{BHL}:

\begin{itemize}
\item[$I_\tau$] $=\{x\in X:\{x\}\in \tau\}$;
\item[$R_\tau$] $=\{x\in X\ms I_\tau: [x,\rightarrow )\in \tau\}$;
\item[$L_\tau$] $=\{x\in X\ms I_\tau: (\leftarrow ,x]\in \tau\}$;
\item[$E_\tau$] $=  X\setminus(I_\tau\cup R_\tau\cup L_\tau)=\{x\in X:$ neither $[x,\to)$ nor $(\from,x]$ is open$\}$ .
\end{itemize}

For a non-empty subset $A\subseteq X$, let $l_A=\inf(A)$ and $u_A=\sup(A)$ which may be gaps in $X$, and define
$[A]=[l_A,u_A] = \{x\in X:l_A\le x\le u_A\}$.
Let $\conv(A)$ denote the convex hull of $A$,
that is, $\conv(A)=\cup\{[p,q]:p,q\in A,p\le q\}$. Clearly $\conv(A)$ is one of the sets
$[l_A,u_A]$, $[l_A,u_A)$, $(l_A,u_A]$, or $(l_A,u_A)$, depending
on which of $l_A$ and $u_A$ belong to $A$. It is easily seen that
if $A$ is open, then so is $\conv(A)$.

It is well-known that metacompact GO-spaces (and more generally
metacompact collectionwise normal spaces) are paracompact \cite[Theorem 5.3.3]{E}.
A monotone version of this result holds for  GO-spaces, and
partially answers our Question~2.6(b) in \cite{PP}.

\begin{theorem}\label{metaparaGO}
Suppose that $(X,\tau,<)$ is a GO-space.
If $X$ is monotonically metacompact
then it has a monotone open locally-finite operator.
\end{theorem}

\begin{proof}
Let $r$ be a monotone metacompactness operator for $X$.
For every open set $U$ let $\mathcal I(U)$ be the family
of all convex components of $U$, where $C\subseteq U$
is a convex component if $C$ is convex, and maximal
with respect to set inclusion (if $C\subseteq C_1 \subseteq U$
where $C_1$ is convex then $C=C_1$).
For each open cover $\mathcal U$ let
$r_1(\mathcal U)=\bigcup\{\mathcal I(U): U\in r(\mathcal U)\}$,
that is, we replace $r(\mathcal U)$ with the cover of all convex
components of elements of $r(\mathcal U)$.
Then $r_1$ is also a monotone metacompactness operator,
and $r_1(\mathcal U)$ consists of convex open sets.
Let $r_2(\mathcal U)=\{U\in r_1(\mathcal U): U$
is maximal in $r_1(\mathcal U)$ with respect to set inclusion$\}$ (where $U\in r_1(\mathcal{U})$ is maximal if
$U\subseteq V\in r_1(\mathcal{U})$ implies $U=V$).
Since $r_1(\mathcal U)$ is point-finite, it contains
no $\subseteq$-strictly increasing sequences. Hence
every element of $r_1(\mathcal U)$ is contained in a
maximal one, and $r_2(\mathcal U)$ covers $X$.
Clearly, $r_2(\mathcal U)$ is point-finite.

If $r_2(\mathcal U)$ were not locally-finite,
then we may fix $p\in X$ and a family
$\mathcal F\subseteq r_2(\mathcal U)$ that is
not locally-finite at $p$ and such that $p\not\in V$,
and either $u_V=\sup V\le p$ for all $V\in \mathcal F$
or $l_V=\inf V\ge p$ for all $V\in \mathcal F$.

Consider the former case (the other being dealt with similarly).
Then $p\in L_\tau\cup E_\tau$, and there is some
$G\in r_2(\mathcal U)$ and $h<p$ such that
$[h,p]\subseteq G$. Every $V\in\mathcal F$ is convex and $l_V=\inf V < h$
(for otherwise $V\subset G$).
There are infinitely many $V\in\mathcal F$ with $h < u_V$,
hence $\mathcal F$ is not point-finite at $h$, a contradiction.
Therefore, $r_2(\mathcal U)$ must be locally-finite.
\end{proof}

We do not know if the assumption that
$X$ is a GO-space could be weakened to the assumption that
$X$ is monotonically normal, or dropped altogether.

\begin{qu}\label{monometamonolf}
If $X$ is monotonically metacompact and monotonically normal,
must it have an (open or not):
(a) monotone locally-finite, or
(b) monotone closure-preserving operator?
\end{qu}

One cannot strengthen the conclusion of Theorem \ref{metaparaGO} to protometrizable spaces. The authors \cite[Example 2.1]{PP} provided an example
of a LOTS with a monotone open locally-finite operator that is not protometrzable.

The following characterization of monotone metacompactness is known
for GO-spaces $(X,\tau,<)$ for which the underlying LOTS
$(X,\lambda,<)$ has a $\sigma$-closed-discrete dense subset.

\begin{theorem}\label{monometa}
\cite[Theorem 1.4]{BHL}, \cite[Theorem 12]{PL}.
Let $(X, \tau ,<)$ be a GO-space whose underlying
LOTS $(X,\lambda,<)$ has a $\sigma$-closed-discrete dense subset.
Then the following are equivalent:

(a) $(X,\tau)$ is monotonically metacompact;

(b) $(X,\tau)$ is monotonically countably metacompact;

(c) the set $R_\tau\cup L_\tau$ is $\sigma$-closed-discrete in $(X, \tau)$;

(d) the set $R_\tau\cup L_\tau$ is $\sigma$-closed-discrete in $(X,\lambda)$.
\end{theorem}

The Michael line $M$ satisfies all conditions above, even though
$M$ itself has no $\sigma$-closed-discrete dense subset.
It is also protometrizable -- equivalent to having a monotone
star-refinement operator \cite[Theorem 1]{GM},  and has a monotone
open locally-finite operator \cite[Corollary 1.7]{PP}.
The following proposition was used
in the proof of Theorem~\ref{monometa}.

\begin{pr}\label{3.8BHL}
\cite[Proposition 3.8]{BHL}, \cite[Proposition 13]{PL}.
Suppose $(X, \tau ,<)$ is a GO-space for which
the underlying LOTS $(X,\lambda,<)$ has
a $\sigma$-closed-discrete dense set.
If $(X,\tau)$ is monotonically countably metacompact,
then $R_\tau\cup L_\tau$ is $\sigma$-closed-discrete
as a subspace of $(X, \tau)$ and as a subspace of $(X,\lambda)$.
\end{pr}

We will prove a similar result for spaces with
a monotone closure-preserving operator $r$,
but first we need to modify $r$.

\begin{lemma}\label{GOopenenough}
Suppose $(X,\tau,<)$ is a GO-space
with a monotone closure-preserving operator $r$.
Then $X$ has a closed convex monotone closure-preserving operator $\bar r$
such that, for every open cover $\mathcal U$ :
\begin{itemize}
\item[(a)] if $x\in R_\tau\cup E_\tau$ then there is
$g_x>x$ and $G\in \bar r(\mathcal U)$ with
$[x,g_x)\subseteq G$, and
\item[(b)] if $x\in L_\tau\cup E_\tau$ then there is
$h_x<x$ and $H\in \bar r(\mathcal U)$ with
$(h_x,x]\subseteq H$.
\end{itemize}
\end{lemma}

\begin{proof}
Given any open cover $\mathcal U$ let
$c(\mathcal U) = \{C\subseteq X : C$ is an open convex subset of X  and $[C]\subseteq U$ for some $U\in \mathcal{U}\}$.
It is easily seen that $c(\mathcal U)$ is an open
cover refining $\mathcal U$, and $c$ is monotone.
Let $\bar r(\mathcal U)=\{[A]: A \in r(c(\mathcal U))\}$.\footnote{
Note that $[A]=\clo_\lambda(\conv(A))$ where $\lambda$ is the
underlying LOTS topology. The above proof would work equally well
if we used $\clo_\tau(\conv(A))$ instead of $\clo_\lambda(\conv(A))$.}

Clearly $\bar r(\mathcal U)$ is a cover of $X$
with closed convex sets and $\bar r$ is monotone.
Also, $\bar r(\mathcal U)$ refines $\mathcal U$ since
if $A\in r(c(\mathcal U))$ then there are
$C\subseteq X$ and $U\in\mathcal U$ such that
$C\in c(\mathcal U)$ and
$A\subseteq C\subseteq [C]\subseteq U$,
hence $[A]\subseteq [C]\subseteq U$.

Suppose $\bar r(\mathcal{U})$ were not
closure-preserving for some $\mathcal{U}$.
Then there is a family $\mathcal A\subseteq r(c(\mathcal U))$
and some
$p\in\cl{\cup\{[A]:A\in\mathcal A\}}\setminus
\cup\{[A]:A\in\mathcal A\}$.
Since each $[A]$ is convex
we have that either $p< l_A$ or $p> u_A$.
We may assume without loss of generality
that $p\in R_\tau\cup E_\tau$ and
$p<l_A$ for all $A\in\mathcal A$.
It is easily seen that $p=\inf\{l_A:A\in\mathcal A\}$
and $p\in\cl{\cup\mathcal A}\setminus
\cup\{\cl{A}:A\in\mathcal A\}$,
contradicting that $r(c(\mathcal U))$ is closure-preserving.
Thus $\bar r(\mathcal{U})$ is closure-preserving
for all open covers $\mathcal{U}$.

To prove (a), fix $x\in R_\tau\cup E_\tau$.
Then $x\in\cl{(x,\to)}$.
For every $y>x$ there is $A_y\in r(c(\mathcal U))$
with $y\in A_y$.
It is enough to show that $l_{A_y}\le x$
for some $y>x$, then we would have
that $[x,u_{A_y})\subseteq [A_y]\in \bar r(\mathcal U)$
and we may pick any $g_x\in(x,u_{A_y})$.
If $l_{A_y}>x$ for each $y>x$ then the family
$\{A_y:y>x\}$ is not closure-preserving
at $x$, contradicting that $r(c(\mathcal U))$ is
closure-preserving.
The proof of (b) is similar.
\end{proof}

The proof of the following lemma is similar to the proof of \cite[Lemma 17]{PL}.

\begin{lemma}\label{rightandleft}
Suppose $(X,\tau,<)$ is a GO-space with
a monotone closure-preserving operator $r$,
$y_n\in R_\tau$ with $y_{n+1}<y_n$ for each $n\in \w$,
and the $y_n$ converge to $y$.
Let $\mathcal{U}_n=\{(\leftarrow ,y_n),[y_n,\rightarrow )\}$.
If $\bar r$ is the monotone operator described in
the proof of Lemma \ref{GOopenenough},
and $G_n\in \bar r(\mathcal{U}_n)$ such that
$y_n\in G_n$, then $\{G_n:n\in\w\}$ is point-finite.
(A similar statement holds for $y_n\in L_\tau$ with
$y_n\nearrow y$ and $\mathcal{U}_n=
\{(\leftarrow ,y_n],(y_n,\rightarrow )\}$.)
\end{lemma}

\begin{proof}
Suppose $\{G_n : n \in \w\}$ were not point-finite.
Taking a subsequence of the $y_n$ we may assume
that there is some $p\in \bigcap \{G_n:n\in\w\}$.
Then $y_n\in G_n\subseteq [ y_n,\to )$, hence $y_n\leq p$ for each $n$.
If $\mathcal{U} = \bigcup \{\mathcal{U}_n : n\in\w\}$,
then $\mathcal{U}_n$ refines $\mathcal{U}$ for each $n$, hence
$\bar r(\mathcal{U}_n)$ refines $\bar r(\mathcal{U})$.
There are $H_n \in \bar r(\mathcal{U})$ and $m_n \geq n$
with $y_n \in G_n \subseteq H_n\subseteq [y_{m_n},\to )$.
Since $y_n$ converges to $y$, the family
$\{H_n : n\in\w \}\subseteq \bar r(\mathcal{U})$
is not closure preserving at $y$.
This contradiction completes the proof.
\end{proof}

The proof of the following proposition is modeled after the proof of \cite[Proposition 13]{PL} and \cite[Proposition 3.8]{BHL} (stated as Proposition~\ref{3.8BHL} here).

\begin{pr}\label{underLOTS}
Suppose $(X, \tau ,<)$ is a GO-space for which
the underlying LOTS $(X,\lambda,<)$ has
a $\sigma$-closed-discrete dense set.
If $(X,\tau)$ has a monotone closure-preserving operator $r$
then $R_\tau\cup L_\tau$ is $\sigma$-closed-discrete
as a subspace of $(X, \tau)$ and as a subspace of $(X,\lambda)$.
\end{pr}

\begin{proof}
Let $D =\bigcup \{D_n: n\in\mathbb{N}\}$ be dense
in $(X,\lambda)$ where each $D_n$ is closed-discrete
in $(X,\lambda)$ (hence also in $(X,\tau)$). It is easily seen that $(X,\tau)$ is first-countable.
By \cite[Lemma 2.4]{BHL}, \cite[Lemma 16]{PL}
it is enough to show that $R_\tau\cup L_\tau$
is $\sigma$-relatively discrete as a subspace of $(X, \tau)$.

For each $p\in R_\tau$, let $\mathcal{U}(p)=\{(\from ,p),[p,\to )\}$.
Let $\bar r$ be the monotone operator described
in the proof of Lemma \ref{GOopenenough}.
Choose $G(p)\in \bar r(\mathcal{U}(p))$ and
$x_p>p$ such that $[p,x_p)\subseteq G(p)$.
There is $n(p)$ such
that $(p, x_p)\cap D_{n(p)}\not = \0$.
Let $R_\tau(n) = \{p\in R_\tau:n(p) =n\}$.
Clearly $R_\tau =\cup\{R_\tau(n) : n\in \w\}$.
We claim that each $R_\tau(n)$ is relatively discrete in $(X,\tau)$.
Suppose not, then there are $n\in \w$,
$p\in R_\tau(n)$ and a sequence $\{p_k: k\in \w\}\subseteq R_\tau(n)$
that converges to $p$ with $p_{k+1} < p_k$ for each $k$.
We may assume that $(p, p_0)\cap D_n=\0$.
Since $[p_k, x_{p_k})\cap D_n \not = \0$,
we have $p_0\in [p_k,x_{p_k})\subseteq G(p_k)\subseteq [p_k,\to )$ for each $k$,
which contradicts Lemma \ref{rightandleft}.

Hence $R_\tau(n)$ is relatively discrete for each $n$,
which shows that $R_\tau$ is $\sigma$-relatively discrete.
Similarly $L_\tau$ is $\sigma$-relatively discrete, which completes the proof.
\end{proof}

Since every monotone open locally-finite operator is both
a monotone metacompactness operator and a
monotone open closure-preserving operator,
Theorem~\ref{metaparaGO}, Theorem~\ref{monometa},
and Proposition~\ref{underLOTS}
allow us to extend Theorem~\ref{monometa} as follows.

\begin{theorem}\label{listmore}
Let $(X, \tau ,<)$ be a GO-space whose underlying
LOTS $(X,\lambda,<)$ has a $\sigma$-closed-discrete dense subset.
Then the following are equivalent:

(i) X has a monotone open locally-finite operator,

(ii) X is monotonically metacompact,

(iii) X has a monotone open closure-preserving operator,

(iv) X has a monotone closure-preserving operator.
\end{theorem}


We do not know if the requirement in Theorem \ref{listmore} that $(X,\l,<)$
has a $\sigma$-closed discrete dense subset is essential.

\begin{qu}\label{q2.8}
If a GO-space $X$ has a monotone (open or otherwise) closure-preserving
operator, must it be monotonically metacompact?
\end{qu}

If the answer to Question \ref{q2.8} is yes, then Theorems 3.3 and 3.8
in the next section would follow from results in \cite{BHL}, \cite{CG}, and \cite{PL}.

\begin{qu} Can one add protometrizable to the list of equivalent conditions
in Theorem \ref{listmore}?
\end{qu}

The following is a variation of our Question 2.6(d) in \cite{PP}
(where ``monotone locally-finite operator'' meant
``monotone open locally-finite operator'').

\begin{qu}\label{maxlfbase}
Does every LOTS $X$ with a monotone
locally-finite (open) operator have
a N\"otherianly locally-finite base (as defined in \cite{PP})?
\end{qu}

\section{Metrization Theorems}\label{esep}

Faber's metrization theorem for GO-spaces was the key to results in \cite{BHL} and \cite{PL} on
the metrization of  monotonically countably metacompact GO-spaces with a $\sigma$-closed-discrete dense subset.

\begin{theorem}\cite[Theorem 3.10]{F} Suppose $(X, \tau, <)$ is a GO-space and $Y \subseteq X$. Then the subspace $(Y, \tau_Y )$ is metrizable if
and only if
\begin{itemize}
\item[(a)] $(Y, \tau_Y )$ has a $\sigma$-closed-discrete dense subset, and
\item[(b)] $R_{\tau_Y}\cup L_{\tau_Y}$ is $\sigma$-closed-discrete in the subspace $(Y, \tau _Y )$.
\end{itemize}
\end{theorem}

By Faber's metrization theorem, to prove that
a GO-space $(X, \tau, <)$ with a $\sigma$-closed discrete dense subset
is metrizable it suffices to show that $R_\tau\cup L_\tau$ is $\sigma$-closed-discrete.

\begin{pr}\label{rldiscrete}
Suppose $(X,\tau,<)$ is a GO-space with a $\sigma$-closed-discrete dense subset. If $X$ has a monotone closure-preserving operator $r$ then $R_\tau\cup L_\tau$ is $\sigma$-closed discrete.
\end{pr}

\begin{proof}
Let $D =\bigcup \{D_n: n\in\mathbb{N}\}$ be dense
in $(X,\tau)$ where each $D_n$ is closed-discrete.
Then $X$ is perfect (and first countable) \cite[Proposition 3.1]{HLP}.
By \cite[Lemma 2.1]{BHL}, it is enough to show that
$R_\tau\cup L_\tau$ is $\sigma$-relatively discrete.

The rest of the proof of Proposition~\ref{underLOTS}
works here without modifications.
\end{proof}

The following theorem immediately follows from the above
proposition and Faber's metrization theorem.

\begin{theorem}\label{gometric}
Suppose $(X, \tau, <)$ is a GO-space with a $\sigma$-closed-discrete dense subset. If $X$ has a
monotone closure-preserving operator, then $(X, \tau )$ is metrizable.
\end{theorem}

By means of a different proof, the first author \cite{P1}
has shown that the Sorgenfrey line does not have
a monotone closure-preserving operator. Since
the Sorgenfrey line is separable and nonmetrizable,
it will not a have monotone closure-preserving
operator by Theorem \ref{gometric}.
(One could also use Theorem~\ref{listmore}.)

\begin{co} The Sorgenfrey line has no
monotone closure-preserving operator.
\end{co}

Every
space $X$ with a (monotone or not) closure-preserving operator must be
paracompact. In particular $\omega_1$ with the order topology
does not have a monotone closure-preserving operator.
The next theorem shows that the compact LOTS
$\omega_1+1$ has no monotone closure-preserving operator either.

\begin{theorem}\label{1stcountable}
Let $X$ be a compact LOTS with a monotone closure-preserving operator $r$.
Then $X$ is first countable.
\end{theorem}

\begin{proof}
If not then we may assume that there is $z \in X$
such that $z\in\cl{(\from,z)}$,
but if $x_n < z$ for each $n\in\omega$ then
$\sup_{n\in\omega}x_n <z$.

For each $x < z$ let $\mathcal{U}(x) =
\{(x,\to)\}\cup\{(\from,y):y<z \}$.
If $x < t < z$ then
$\mathcal{U}(t)$ refines $\mathcal{U}(x)$.
Fix $x_0 < z$ and let $\mathcal{A}(x_0)=
\{A\in r(\mathcal{U}(x_0)) : x_0\in A\}$.
If $A\in \mathcal{A}(x_0)$, then
$A\subseteq (\from,y)$
for some $y < z$.
Then $u_A=\sup(A)\le y < z$.
Also, $\sup\{u_A:
A\in \mathcal{A}(x_0)\}=u_B < z$
for some $B\in \mathcal{A}(x_0)$.
Indeed, otherwise we could take
$A_k\in \mathcal{A}(x_0)$ with
$u_{A_k}<u_{A_{k+1}}$ (for all $k\in\omega$)
and then the family $\{A_k : k\in \omega \}$
would not be closure-preserving at
$\sup\{u_{A_k}:k\in \omega \}$.

By induction pick $x_n < z$ with
$x_{n+1} > \sup\{u_A :
A\in\mathcal{A}(x_n)\}$,
where
$\mathcal{A}(x_n)=\{A\in r(\mathcal{U}(x_n)):
x_n\in A\}$.
Clearly $x_n < x_{n+1}$.
Let $t = \sup_{n\in\omega}x_n$, then $t < z$.
For each $n$ pick $C_n\in r(\mathcal{U}(t))$
such that $x_n\in C_n$.
Since $\mathcal U(t)$ refines each
$\mathcal U(x_n)$, there is
$A_n\in \mathcal{A}(x_n)$ with
$C_n\subseteq A_n$.  Hence
$x_n\leq \sup(C_n)\leq \sup(A_n) < x_{n+1}$.
It follows that the family $\{C_n : n\in \omega \}$
is not closure-preserving at $t$, a contradiction.
\end{proof}

We do not know if in the above theorem we may conclude that $X$ is metrizable.
We will show that if the monotone closure-preserving operator $r$ is open, then
the answer is yes.
Again, we modify the monotone operator.

\begin{lemma}\label{convex}
Suppose $(X,\tau,<)$ is a GO-space with
a monotone closure-preserving operator $r$.
Then $X$ has a convex monotone closure-preserving operator $r_1$,
such that 
if $r$ is an open operator, then so is $r_1$.
\end{lemma}

\begin{proof}
Let $r_1(\mathcal{U}) = \{\conv(A) : A\in r(c(\mathcal U))\}$,
where $c$ is defined as in the proof of Lemma \ref{GOopenenough}.
The easy verification that $r_1$ is the desired monotone operator is left to the reader.
\end{proof}

\begin{lemma}\label{finite}
Suppose $(X,<)$ is a compact LOTS with
a monotone closure-preserving operator $r$.
Then $X$ has a convex monotone closure-preserving operator $r_2$
such that $r_2(\mathcal{U})$ is finite,
for every open cover $\mathcal{U}$. 
If $r$ is open, then so is $r_2$.
\end{lemma}

\begin{proof}
If $r_1$ is the convex monotone closure-preserving operator
operator defined in Lemma \ref{convex},
let $r_2(\mathcal{U})=\{U\in r_1(\mathcal{U}): U$
is maximal in $r_1(\mathcal{U})$ with respect to inclusion$\}$.
Every element of $r_1(\mathcal{U})$ is contained in
a $\subseteq$-maximal one, for otherwise we could find
a $\subseteq$-strictly increasing chain
$\mathcal J=\{J_n:n\in\omega\}\subseteq r_1(\mathcal{U})$ with
$J_n\subset J_{n+1}$ for all $n$, but then
$\mathcal J$ would not be closure-preserving at either
$\sup\{u_{J_n}:n\in\omega\}$ or at $\inf\{l_{J_n}:n\in\omega\}$.
(Since the $J_n$ are convex and
$\subset$-increasing, there are infinitely many $n$
for which either $l_{{J_{n+1}}}< l_{J_n}$ or $u_{{J_n}}< u_{J_{n+1}}$.)
Hence $r_2(\mathcal{U})$ covers $X$.

Given any nonempty (usually convex) $A,B\subseteq X$
define $A\ll B$ provided that
either there is $a\in A$ with $a<b$ for all $b\in B$, or there
is $b\in B$ with $a<b$ for all $a\in A$.
Since each element of $r_2(\mathcal{U})$ is convex and $\subseteq$-maximal it
follows that $\ll$ totally orders $r_2(\mathcal{U})$ (i.e. every two
distinct elements of $r_2(\mathcal{U})$ are $\ll$-comparable).
If $r_2(\mathcal{U})$ were infinite for some open cover $\mathcal U$
then we could find a family
$\mathcal I=\{I_n:n\in\omega\}\subseteq
r_2(\mathcal{U})\subseteq r_1(\mathcal{U})$ with
either $I_n\ll I_{n+1}$ for all $n$, or $I_{n+1}\ll I_n$ for all $n$.
In the former case $\mathcal I$
is not closure-preserving
at $\sup\{u_{I_n}:n\in\omega\}$, and in the latter case
at $\inf\{l_{I_n}:n\in\omega\}$. This contradiction shows that
$r_2(\mathcal{U})$ is finite for all open covers $\mathcal U$.
Clearly $r_2$ is monotone, and if $r$ is open, so is $r_2$.
\end{proof}

\begin{theorem}\label{compLOTSopen}
Suppose that $X$ is a compact LOTS with a
monotone open closure-preserving operator $r$.
Then $X$ is metrizable.
\end{theorem}

\begin{proof}
By Lemma \ref{finite}, $X$ is monotonically compact
(i.e. it has a monotone operator $r_2$ such that $r_2(\mathcal U)$
is a finite {\em open} refinement covering $X$,
for every open cover $\mathcal U$).
Hence $X$ is metrizable \cite[Theorem 4.1]{Gmc} (for LOTS). More generally, see \cite{CG,GmcH,Pmc}.
\end{proof}

If Question \ref{opencp} has a positive answer, then the answer to the following question would also be positive.

\begin{qu}\label{compactmet}
If $X$ is a compact LOTS with a monotone
closure-preserving operator $r$,
must $X$ be metrizable?
\end{qu}

If $E_\tau$ is empty, then we have the following partial answer
to Question \ref{opencp}.

\begin{theorem}\label{GOopen}
Suppose $(X,\tau,<)$ is a GO-space
with a monotone closure-preserving operator $r$.
Then there is a convex open monotone operator
$\mathring r$ such that $\mathring r(\mathcal U)$
is a closure-preserving family and
$X\setminus \cup\mathring r(\mathcal U)
\subseteq E_\tau$,
for every open cover $\mathcal U$.
If $E_\tau=\0$, then $\mathring r$
is a convex monotone open closure-preserving operator.
\end{theorem}

\begin{proof}
Let $c$ and $\bar r$ be the operators described in
the proof of Lemma \ref{GOopenenough}.
Let $\mathring r(\mathcal U)=
\{\Int[A]: A \in r(c(\mathcal U))\} =
\{\Int(K): K \in \bar r(\mathcal U)\}$.
Clearly $\mathring r(\mathcal U)$
is an open family refining $\bar r(\mathcal U)$
(and hence also $\mathcal U$)
and $\mathring r$ is monotone.
The proof that $\mathring r(\mathcal U)$
is a closure-preserving family is similar to the
proof for $\bar r(\mathcal U)$ and is left to the reader.

We show that $\cup\mathring r(\mathcal U)\supseteq X\setminus E_\tau$.
If $x\in I_\tau$, then $x\in\Int[A]\in\mathring r(\mathcal U)$
whenever $x\in A \in r(c(\mathcal U))$.
If $x\in R_\tau$ then by Lemma \ref{GOopenenough}
($a$), there is some $G\in \bar r(\mathcal U)$ with
$x\in \Int(G)$. The case $x\in L_\tau$ is similar,
which completes the proof.
\end{proof}

In the special case when $E_\tau$ is finite, Theorem \ref{GOopen} allows us to remove the requirement in
Theorem~\ref{compLOTSopen} that the operator $r$ is open.

\begin{theorem}\label{compLOTSfinite}
Suppose that $X$ is a compact LOTS with a
monotone closure-preserving operator $r$.
If $E_\tau$ is finite, then $X$ is metrizable.
\end{theorem}

\begin{proof}
Let $\mathring r(\mathcal U)$ be as described
in the proof of the preceding theorem,
and let $\hat r(\mathcal U)$ be the family of
$\subseteq$-maximal elements of $\mathring r(\mathcal U)$.
It is easily seen
(using the ideas in the proof of Lemma~\ref{finite})
that $\cup\hat r(\mathcal U)= \cup\mathring r(\mathcal U)$
and that $\hat r(\mathcal U)$ is finite,
for any open cover $\mathcal U$.
If $E_\tau=\emptyset$ then we are done as $\hat r$ shows
that $X$ is monotonically compact.

If $E_\tau\not=\emptyset$ then
(using Theorem~\ref{1stcountable})
for each $x\in E_\tau$ fix a $\subseteq$-decreasing
local base $\mathcal B_x=\{B_n(x):n\in\omega\}$
(i.e. $B_{n+1}(x)\subset B_n(x)$ for all $n$).
Given any open cover
$\mathcal U$ let $V_x(\mathcal U)$ be the
$\subseteq$-maximal element of $\mathcal B_x$
that is contained in some open set $U\in\mathcal U$
(i.e. $V_x(\mathcal U)=B_n(x)$ where
$n$ is smallest such that there is
$U\in\mathcal U$ with $B_n(x)\subseteq U$).
Let $\mathcal F(\mathcal U)=
\{V_x(\mathcal U):x\in E_\tau\}$.
Since $E_\tau$ is finite, the operator
$\tilde r(\mathcal U)=\hat r(\mathcal U)\cup
\mathcal F(\mathcal U)$
shows that $X$ is monotonically compact,
and hence metrizable.
\end{proof}

\begin{co}\label{arrow}
The Alexandroff double arrow is a compact first-countable, hereditarily Lindel\"of LOTS that has no monotone closure-preserving operator.
\end{co}

\end{document}